\documentclass{amsproc}
\usepackage{amsmath}
\usepackage{amsfonts}
\usepackage{amssymb}
\usepackage{hyperref}
\hypersetup{colorlinks,
citecolor=black,
filecolor=black,
linkcolor=black,
urlcolor=black}
\usepackage{graphics,epstopdf}
\usepackage[pdftex]{graphicx}
\usepackage{newlfont}\newlength{\defbaselineskip}
\usepackage[left=1in,right=1in,top=1in,bottom=0.8in,footskip=0.25in]{geometry}
\usepackage{setspace}
\allowdisplaybreaks
\newtheorem{theorem}{Theorem}[section]
\newtheorem{example}{Example}[section]
\newtheorem{lemma}{Lemma}[section]

\newtheorem{remark}{Remark}[section]

\numberwithin{equation}{section}

\usepackage{cite}

\usepackage{fancyhdr}
\usepackage{graphics}
\usepackage{color,colortbl}
\usepackage{rotating}
\usepackage{fancybox}
\usepackage[table]{xcolor}

\begin{document}
\title{Further approximations of Durrmeyer modification of Sz\'asz-Mirakjan operators 
}
\maketitle
\begin{center}
{\bf Rishikesh Yadav$^{1,\dag}$,  Ramakanta Meher$^{1,\star}$,  Vishnu Narayan Mishra$^{2,\circledast}$}\\
$^{1}$Applied Mathematics and Humanities Department,
Sardar Vallabhbhai National Institute of Technology Surat, Surat-395 007 (Gujarat), India.\\
$^{2}$Department of Mathematics, Indira Gandhi National Tribal University, Lalpur, Amarkantak-484 887, Anuppur, Madhya Pradesh, India\\
\end{center}
\begin{center}
$^\dag$rishikesh2506@gmail.com,  $^\star$meher\_ramakanta@yahoo.com,
 $^\circledast$vishnunarayanmishra@gmail.com
\end{center}

\vskip0.5in

\begin{abstract}
The main purpose of this paper is to determine the approximations of Durrmeyer modification of Sz\'asz-Mirakjan operators, defined by Mishra et al. (Boll. Unione Mat. Ital. (2016) 8(4):297-305). We estimate the order of approximation of the operators for the functions belonging to the different spaces.  Here, the rate of convergence of the said operators is established by means of the function with derivative of the bounded variation. At last, the graphical analysis is discussed to support the approximation results of the operators.
\end{abstract}

\subjclass \textbf{MSC 2010}: {41A25, 41A35, 41A36}.

\section{Introduction}
In 1944, Mirakjan \cite{GMM2} and 1950, Sz\'asz \cite{SO} introduced operators on unbounded interval $[0,\infty)$, known as Sz\'asz-Mirakjan operators defined by
\begin{eqnarray}\label{sm1}
S_n(g;x)=\sum\limits_{j=0}^\infty s_{n,j}(x) g\left(\frac{j}{n} \right),
\end{eqnarray}
where $s_{n,j}=e^{-nx}\frac{(nx)^{j}}{j!}$, $g\in C_2[0,\infty)=\{g\in C[0,\infty):\underset{x\to \infty}\lim\frac{f(x)}{1+x^2}\; \text{exists and finite} \}$, $x\geq 0$ and for all $n\in\mathbb{N}$.\\

An integral modification of the above operators (\ref{sm1}) can bee seen in \cite{BPL} to estimate the approximation results for the integrable function. The important properties including global results, local results, simultaneous approximation, convergence properties etc. have been studied with the above operators and their modifications in various studies (see \cite{AU,GDM,MG3,MKMD,MKK,MKM}). One of them, an interesting modification was the Durrmeyer modification of the Sz\'asz-Mirakjan operators and which can can be written as:
  \begin{eqnarray}\label{do}
D_n(g;x)=\sum\limits_{j=0}^\infty s_{n,j}(x) \int\limits_0^{\infty} s_{n,j}(t) g(t) dt,
\end{eqnarray}
 seen in \cite{MS}. Also, another modification into Stancu variant appeared in \cite{RBG} of the above operators (\ref{do}) and related properties like density, direct results as well as Voronovskaya type theorem are studied. 

A natural generalization is carried out for the above operators (\ref{do}) in \cite{MGN} by Mishra et al. for the study of simultaneous approximation, like

\begin{eqnarray}\label{rb}
B_n^*(g;x)=u_n\sum\limits_{j=0}^\infty s_{u_n,j}(x)\int\limits_0^{\infty} s_{u_n,j}(t) g(t) dt,
\end{eqnarray}
where $s_{u_n,j}(x)=e^{-u_nx}\frac{(u_nx)^{j}}{j!}$ by considering the sequence $u_n$ is strictly increasing of positive real number as well as $u_n\to\infty$ as $n\to\infty$ with $u_1=1$.


Our main motive is to study the approximation properties of the proposed operators (\ref{rb}) for the functions belonging from different spaces. 
The important properties of the above proposed operators (\ref{rb}) are studied by authors which can also be applied to the operators defined by (\ref{do}).

In order to study the operators (\ref{rb}), we divide the paper into sections. Section second contains preliminaries results, which are used to prove the main theorems. Section third deals with the approximation properties of the operators for the function belong to the different spaces of functions classes. In section fourth, the rate of convergence is estimated of the operators for the functions with derivative of bounded variation.  At last, we present the graphical and numerical representation for the operators in order to show the convergence of the operators.

%
\section{Preliminary}
This section contains the basic properties of the defined operators (\ref{rb}). In order to prove approximations properties, we need basic lemmas.
\begin{lemma}
For all $x\geq 0$ and $n\in\mathbb{N}$, we have
\begin{eqnarray*}
B_n^*(1;x)&=&1\\
B_n^*(t;x)&=& \frac{1}{u_n}+x\\
B_n^*(t^2;x)&=& \frac{2+4 x u_n+x^2 u_n^2}{u_n^2}\\
B_n^*(t^3;x)&=&\frac{6+18 x u_n+9 x^2 u_n^2+x^3 u_n^3}{u_n^3}
\end{eqnarray*}
\end{lemma}

\begin{proof}
We can easily proof the above parts of the lemma, so we omit the proof.
\end{proof}

\begin{lemma}\label{lem1}
Consider the function  $g$ is integrable, continuous, bounded on given interval $[0,\infty)$, then the central moments can be obtained as:
\begin{eqnarray}
\Omega_{n,m}=u_n\sum\limits_{j=0}^\infty s_{u_n,j}(x)\int\limits_0^{\infty} s_{u_n,j}(t) (t-x)^m dt,
\end{eqnarray}
where $m=0,1,2,\ldots$. So for $m=0,1$, we get the the central moments as follows:
\begin{eqnarray}
\Omega_{n,0}=1, \Omega_{n,1}=\frac{1}{u_n},
\end{eqnarray}
in general, we have
\begin{eqnarray}
u_n\Omega_{n,m+1}=x\left(\Omega_{n,m}'+2m\Omega_{n,m-1}+(1+m)\Omega_{n,m} \right),
\end{eqnarray}
this lead us to 
\begin{eqnarray}
\Omega_{n,m}=O\left(u_n^{-\left[\frac{m+1}{2}\right]} \right).
\end{eqnarray}
\end{lemma}
\begin{lemma}
Let the function $g$ be the continuous and bounded on $[0,\infty)$ endowed with supremum norm $\|g(x)\|=\underset{x\geq 0}\sup |g|$ then, we have 
\begin{eqnarray}
|B_n^*(g;x)|\leq \|g\|.
\end{eqnarray}
\end{lemma}
\begin{remark}
For second order central moment, it can be written as
\begin{eqnarray}\label{rem}
\Omega_{n,2}=\frac{2 (1+u_n x)}{u_n^2}=\frac{2}{u_n}\left(x+\frac{1}{u_n} \right)= \frac{2}{u_n}\zeta_n^2(x),
\end{eqnarray}
where $\zeta_n^2(x)=\left(x +\frac{1}{u_n}\right).$
\end{remark}
\section{Approximation properties}
Consider $C_B[0,\infty)$ be the space of all continuous and bounded function defined on $[0,\infty)$, endowed with supremum norm $\|g\|=\underset{x\geq 0}\sup |g(x)|$, also let for any $\delta>0$ 
\begin{eqnarray}
K_2(g;\delta)=\underset{f\in E}\inf\{\|g-f\|+\delta\|f''\|\}
\end{eqnarray}
be the Peetre's $K$-functional, where $E=\{f\in C_B[0,\infty): f', f''\in C_B[0,\infty) \}$. Also a relation can be seen  for which there exists a positive constant $M$ such that:
\begin{eqnarray}\label{ine1}
K_2(g;\delta)\leq M \omega_2 (g,\sqrt{\delta}),~~\delta>0,
\end{eqnarray}
where $\omega_2 (g,\sqrt{\delta})$ is second order modulus of smoothness the function $g\in C_B[0,\infty)$, which is defined by:
\begin{eqnarray}
\omega_2(g,\delta)=\sup\{g(x+h)-2g(x)+g(x-h):x,x\pm h\in [0,\infty), 0\leq h\leq \delta\},
\end{eqnarray}
also usual modulus of continuity can be defined for the function $g\in C_B[0,\infty)$ as follows:
\begin{eqnarray}
\omega(g,\delta)=\{g(y)-g(x):x,y\in [0,\infty), |y-x|\leq\delta, \delta>0\}. 
\end{eqnarray}
\begin{theorem}
Consider $g\in C_B[0,\infty)$ and for all $x\geq0$ then there exists a positive constant $M$ such that
\begin{eqnarray}
|B_n^*(g;x)-g(x)|\leq C\omega_2\left(g,\frac{\sqrt{\delta_n}}{2}\right)+\omega\left(g,\gamma_n\right),
\end{eqnarray}
where $\delta_n=\tilde{B}_{n}^*((t-x)^2;x)+\frac{1}{u_n^2}$ and $\gamma_n=\tilde{B}_{n}^*((t-x);x).$
\end{theorem}
\begin{proof}
Here, we consider the auxiliary operators such that
\begin{eqnarray}\label{NO2}
{\tilde{S}}_{n}^*(g;x)=B_n^*(g;x)-g\left(\frac{1}{u_n}+x\right) +g(x).
\end{eqnarray}
Let $f\in E$, $x\geq 0$ the using Taylor's formula, we have
\begin{eqnarray}
f(t)-f(x)=(t-x)f'(x)+\int\limits_0^t(t-v)f''(v)dv
\end{eqnarray}
Applying the operators $\tilde{B}_{n}^*$ on the both sides to the above expression, it yields:
\begin{eqnarray}\label{e1}
\tilde{B}_{n}^*(f;x)-f(x)&=&f'(x)\tilde{B}_{n}^*(t-x;x)+\tilde{B}_{n}^*\left(\int\limits_x^t(t-v)f''(v)dv \right)\nonumber\\
&=& \tilde{B}_{n}^*\left(\int\limits_x^t(t-v)f''(v)dv \right)\nonumber\\
&=& S_{n}^*\left(\int\limits_x^t(t-v)f''(v)dv \right)-\left(\int\limits_x^{\left(\frac{1}{u_n}+x\right)}\left(\frac{1}{u_n}+x -v\right)f''(v) dv \right).
\end{eqnarray}
Here, the inequalities are as:
\begin{eqnarray}\label{i1}
\left|\int\limits_x^t(t-v)f''(v)dv \right|\leq (t-x)^2\|f''\|
\end{eqnarray}
and 
\begin{eqnarray}\label{i2}
\left|\int\limits_x^{\left(\frac{1}{u_n}+x\right)}\left(\frac{1}{u_n}+x -v\right)f''(v) dv \right|\leq \frac{1}{u_n^2} \|f''\|.
\end{eqnarray}
By considering the above inequalities (\ref{i1}, \ref{i2}) and with (\ref{e1}), we obtain
\begin{eqnarray}\label{e2}
\tilde{B}_{n}^*(f;x)-f(x)&=& \left\{\tilde{B}_{n}^*((t-x)^2;x)+\frac{1}{u_n^2} \right\}\|f''\|\\
&=& \delta_n \|f''\|.
\end{eqnarray}
Also, $|S_{n}^*(g;x)|\leq \|g\|$. Using this property, we get

\begin{eqnarray*}
|S_{n}^*(g;x))-g(x)| & \leq & |\tilde{B}_{n}^*(g-f;x)-(g-f)(x)|+|\tilde{B}_{n}^*(f;x)-f(x)| \\ &&+ \left|g\left(\frac{1}{u_n}+x\right)-g(x)\right| \\ & \leq & 4\|g-f\| +| \tilde{B}_{n}^*(f;x)-f(x)|+ \left|g\left(\frac{1}{u_n}+x\right)-g(x)\right|,
\end{eqnarray*}
using  (\ref{e2}) and with the help of modulus of continuity, we obtain
\begin{eqnarray*}
|S_{n}^*(g;x)-g(x)| & \leq & 4\|g-f\| +\delta_n\|f''\|+ \omega\left(g,\gamma_n\right).
\end{eqnarray*}
Taking the infimum for all $f\in E$ on the right hand side and by relation (\ref{ine1}), we get

\begin{eqnarray*}
|S_{n}^*(g;x)-g(x)| & \leq & 4K_2\left(g;\frac{1}{4}\delta_n\right)+ \omega\left(g,\gamma_n\right)\\ & \leq & C\omega_2\left(g,\frac{\sqrt{\delta_n}}{2}\right)+\omega\left(g,\gamma_n\right).
\end{eqnarray*}
Thus, the proof is completed.
\end{proof}


Now, we estimate the approximation of the defined operators (\ref{rb}), by new type of Lipschitz maximal function with order $s\in(0,1]$,  defined by Lenze \cite{LB} as
\begin{eqnarray}\label{eq8}
\tau_s(g,x)=\underset{x,t\geq 0}\sup \frac{|g(t)-g(x)|}{|t-x|^s},~~t\neq x. 
\end{eqnarray}
Using Lipschitz maximal function definition, we have following theorem.
\begin{theorem}
For any $g\in C_B[0,\infty)$ with $s\in(0,1]$ then one can obtain
\begin{eqnarray*}
\left|B_n^*(g;x)-g(x)\right| &\leq  \tau_s(g,x)\left(\Omega_{n,2}\right)^{\frac{s}{2}}.
\end{eqnarray*}
\end{theorem}
\begin{proof}
By equation (\ref{eq8}), we can write
\begin{eqnarray*}
\left|B_n^*(g;x)-g(x)\right| &\leq \tau_s(g,x)B_n^*(|t-x|^s;x).
\end{eqnarray*}
Using, H$\ddot{\text{o}}$lder's inequality with $j=\frac{2}{s}$, $l=\frac{2}{2-s}$, one can get
\begin{eqnarray*}
\left|B_n^*(g;x)-g(x)\right| &\leq & \tau_s(g,x)\left(B_n^*(g;x)((t-x)^2;x)\right)^{\frac{s}{2}}=\tau_s(f,x)\left(\Omega_{n,2}\right)^{\frac{s}{2}}.
\end{eqnarray*}
\end{proof}

Next theorem is based on modified Lipschitz type spaces \cite{OMA} and this spaces is defined by 
\begin{eqnarray*}
Lip_M^{m_1,m_2}(s)=\Bigg\{ g\in C_B[0,\infty):|g(l_1)-g(l_2)|\leq M\frac{|l_1-l_2|^s}{\left(l_1+l_2^2m_1+l_2 m_2\right)^{\frac{s}{2}}},~~\text{where}~l_1, l_2\geq0 ~\text{are~variables},~s\in(0,1] \Bigg\}
\end{eqnarray*}
and $m_1, m_2$ are the fixed numbers and $M>0$ is a constant.
\begin{theorem}
For $g\in Lip_M^{m_1,m_2}(s)$ and $0<s\leq 1$, an inequality holds:
\begin{eqnarray*}
\left|B_n^*(g;x)-g(x)\right| &\leq & M\left(\frac{\Omega_{n,2}}{x(xm_1+m_2)}\right)^{\frac{s}{2}},~~M>0, x\in[0,\infty).
\end{eqnarray*}
\end{theorem}

\begin{proof}
We have $s\in(0,1]$ and in order to prove the above theorem,  we discuss the cases on $s$.\\

\textbf{Case 1.} if we consider $s=1$ then for all $t,x\geq 0$, we can observe that  $\frac{1}{t+x^2m_1+xm_2)}\leq \frac{1}{x(xm_1+m_2)}$ then 
\begin{eqnarray*}
\left|B_n^*(g;x)-g(x)\right| &\leq & B_n^*(|g(t)-g(x)|;x)\\
&\leq & M B_n^*\left(\frac{|t-x|}{\left(t+x^2m_1+xm_2\right)^{\frac{1}{2}}};x\right)\\
&\leq & \frac{M}{\left(x(xm_1+m_2)\right)^{\frac{1}{2}}}B_n^*(|t-x|;x) \\
&\leq & \frac{M}{\left(x(xa_1+a_2)\right)^{\frac{1}{2}}}\left(\Omega_{n,2}\right)^{\frac{1}{2}}\\
&\leq & M\left(\frac{\Omega_{n,2}}{x(xm_1+m_2)}\right)^{\frac{1}{2}}.
\end{eqnarray*}
\textbf{Case 2.} for $s\in (0,1)$ then using  H$\ddot{\text{o}}$lder inequality with $p=\frac{2}{s}, q=\frac{2}{2-s}$, we get
\begin{eqnarray*}
\left|B_n^*(g;x)-g(x)\right| &\leq & \left(B_n^*(|g(t)-g(x)|^{\frac{2}{s}};x)\right)^{\frac{s}{2}}\leq M B_n^*\left(\frac{|t-x|^{2}}{\left(t+x^2m_1+xm_2\right)};x\right)^{\frac{s}{2}}\\
&\leq & M B_n^*\left(\frac{|t-x|^{2}}{\left(x(xm_1+m_2)\right)};x\right)^{\frac{s}{2}}\\
&\leq & M\left(\frac{\Omega_{n,2}}{x(xm_1+m_2)}\right)^{\frac{s}{2}}.
\end{eqnarray*}
This complete the proof.
\end{proof}

\begin{theorem}\label{th1}
For the function $g$ which is continuous and bounded on $[0,\infty)$, the convergence of the operators can be obtained as:
\begin{eqnarray}
\underset{n\to\infty}\lim B_n^*(g;x)=g(x), 
\end{eqnarray}
uniformly on any compact interval of $[0,\infty)$.
\end{theorem}
\begin{proof}
Using Bohman-Korovkin theorem, we can get our required result. Since $\underset{n\to\infty} \lim B_n^*(1;x)\to1$, $\underset{n\to\infty} \lim B_n^*(t;x)\to x$, $\underset{n\to\infty} \lim B_n^*(t^2;x)\to x^2$ and hence the proposed operators $B_n^*(g;x)$  converge uniformly to the function $g(x)$ on any compact interval of $[0,\infty)$. 
\end{proof}

\section{Rate of convergence by means of the function with derivative of bounded variation}
This section consists the rate of convergence by means of the function with derivative of bounded variation. Let $DBV[0,\infty)$ be the set of all class of function having derivative of bounded variation on every compact interval of $[0,\infty)$. The following representation for the function $g\in DBV[0,\infty)$, is as follows:
\begin{eqnarray}
g(x)=\int\limits_0^x h(t)dt+g(0),
\end{eqnarray}
where $h(t)$ is a function with derivative of bounded variation on any compact interval of $[0,\infty)$. 

For investigation of the convergence of the above operators (\ref{rb}) to the function with derivative of bounded variation, we rewrite (\ref{rb}) as follows:
\begin{eqnarray}\label{new}
B_n^*(g;x)=\int_0^\infty Y_n(x,t)g(t)dt,
\end{eqnarray}
 where 
 \begin{eqnarray*}
 Y_n(x,t)=u_n\sum\limits_{j=0}^\infty s_{u_n,j}(x)\ s_{u_n,j}(t).
 \end{eqnarray*}
Such type of properties have been studied by researchers using various operators  (see\cite{IN1,KH,Y1,Y2,Y3}.)
\begin{lemma}\label{lem2}
For sufficiently large value of $n$ and for all $x\geq 0$, we have 
\begin{enumerate}
\item{} $I_n(x,t)=\int\limits_0^y  Y_n(x,t)dt\leq \frac{2}{(x-y)^2u_n}\zeta_n^2(x),~~~0\leq y< x,$
\item{} $1-I_n(x,t)=  \int\limits_z^\infty  Y_n(x,t)dt \leq \frac{2}{(z-x)^2u_n}\zeta_n^2(x),~~~x\leq z<\infty.$
\end{enumerate}
\end{lemma}
\begin{proof}
Using the Lemma \ref{lem1} and since the value of $n$ is sufficiently large, so we have
\begin{eqnarray*}
I_n(x,t)&=&\int\limits_0^y  Y_n(x,t)dt \leq \int\limits_0^y \left(\frac{(x-t)^2}{(x-y)^2} \right) Y_n(x,t)dt\\
&=& \frac{2}{(x-y)^2u_n}\zeta_n^2(x).
\end{eqnarray*}
 Similarly, we can prove other inequality.
 \end{proof}

\begin{theorem}
Let $g\in DBV[0,\infty)$, then for all $x\geq0$, an upper bound of the operators to the function can be as:
\begin{eqnarray*}
|B_n^*(g;x)-g(x)|&\leq &  \frac{1}{2u_n}|g'(x+)+g'(x-)|+ \sqrt{\frac{1}{2u_n}}|g'(x+)-g'(x-)|\zeta_n(x)\\
&& +\frac{2\zeta_n^2(x)}{xu_n} \sum\limits_{j=0}^{[\sqrt{u_n}]}\left(V_{x-\frac{x}{j}}^tg'_x \right)+\frac{x}{\sqrt{u_n}}\left(V_{x-\frac{x}{\sqrt{u_n}}}^xg'_x \right)\\
 && + \frac{x}{\sqrt{u_n}} V_x^{x+\frac{x}{\sqrt{u_n}}}(g'_x)+ \frac{2\zeta_n^2(x)}{xu_n} \sum\limits_{j=0}^{[\sqrt{u_n}]} V_x^{x+\frac{x}{j}}(g'_x),
\end{eqnarray*}
where 
\begin{eqnarray}\label{eq5}
g_x(t) &= &  
\begin{cases}
    g(t)-g(x-),& 0\leq t<x,\\
    0,& t=x,\\
    g(t)-g(x+), &  x<t<\infty
\end{cases} 
\end{eqnarray}
be an auxiliary operator and $ V_a^b g(x)$ denotes the total variation of the function $g(x)$ on $[a,b]$.
\end{theorem}

\begin{proof}
Since, $B_n^*(1;x)=1$ and hence, one can write
\begin{eqnarray*}
B_n^*(g;x)-g(x)&=&\int_0^\infty (g(t)-g(x))Y_n(x,t)dt\\
&=&  \int_0^\infty Y_n(x,t)dt \int\limits_x^t g'(u)du.
\end{eqnarray*}
\end{proof}
Now, for $g\in DBV[0,\infty)$, we can write as:
\begin{eqnarray*}
g'(u)=\frac{1}{2}(g'(x+)+g'(x-))+g_x'(u)+\frac{1}{2}(g'(x+)+g'(x-))\left(sgn(u-x) \right)+\eta(u)\left(g'(u)-\frac{1}{2}(g'(x+)+g'(x-)) \right),
\end{eqnarray*}

where 
\begin{eqnarray}\label{bv2}
\eta(u)=
\begin{cases}
1 & u=x\\
0 & u\neq x.
\end{cases}
\end{eqnarray}
And then, one can show
\begin{eqnarray}\label{bv4}
\int\limits_0^\infty Y_n(x,t)\int\limits_x^t\left(\eta(u)\{g'(u)-\frac{1}{2}(g'(x+)+g'(x-))\}du \right)dt=0.
\end{eqnarray}
Using (\ref{new}), we can get

\begin{eqnarray}
\int\limits_0^\infty Y_n(x,t)\left(\int\limits_x^s\frac{1}{2}(g'(x+)+g'(x-))~du\right)dt &=&\frac{1}{2}(g'(x+)+g'(x-)) \int\limits_0^\infty Y_n(x,t)(t-x)~dt\nonumber\\&=& \frac{1}{2}(g'(x+)+g'(x-)) \Omega_{n,1}.
\end{eqnarray}

And 

\begin{eqnarray}\label{bv6}
\left|\int\limits_0^\infty Y_n(x,t)\left(\frac{1}{2}\int\limits_x^t(g'(x+)-g'(x-))\text{sgn}(u-x)~du\right)dt\right| &\leq & \frac{1}{2}|(g'(x+)-g'(x-))|\int\limits_0^\infty Y_n(x,t)|t-x|~dt\nonumber\\
&\leq & \frac{1}{2}|(g'(x+)-g'(x-))| \int\limits_0^\infty |t-x|Y_n(x,t) dt\nonumber\\
&\leq & \frac{1}{2}|(g'(x+)-g'(x-)|\left(\Omega_{n,2} \right)^{\frac{1}{2}}
\end{eqnarray}

Using (\ref{rem}), we get:

\begin{eqnarray}\label{bv7}
|B_n^*(g;x)-g(x)|&\leq &  \frac{1}{2}|g'(x+)+g'(x-)|\Omega_{n,1}+ \frac{1}{2}|g'(x+)-g'(x-)|\sqrt{\frac{2}{u_n}}\zeta_n(x)\nonumber\\
&&+ \left|\int\limits_0^\infty Y_n(x,t)\left(\frac{1}{2}\int\limits_x^s(g'_x(u))~du\right)dt\right|
\end{eqnarray}

Here,
 
   \begin{eqnarray}\label{bv8}
   \nonumber \int\limits_0^\infty Y_n(x,t)\left(\int\limits_x^s(g'_x(u))~du\right)dt &=& \int\limits_0^x Y_n(x,t)\left(\int\limits_x^s(g'_x(u))~du\right)dt+\int\limits_x^\infty Y_n(x,t)(x,t)\left(\int\limits_x^t(g'_x(u))~du\right)dt\\
   &=&P_1+P_2,
   \end{eqnarray}

where 

\begin{eqnarray}
P_1 &=& \int\limits_0^x \left(\int\limits_x^t(g'_x(u))~du\right)\frac{\partial}{\partial t}(I_n(x,t))dt\nonumber\\
&=& \int\limits_0^x g'_x(t)I_n(x,t)dt\nonumber\\
&=& \int\limits_0^y g'_x(t)I_n(x,t)dt+\int\limits_y^x g'_x(t)I_n(x,t)dt
\end{eqnarray}

Here, we consider $y=x-\frac{x}{\sqrt{u_n}}$ then by above equality, one can write

\begin{eqnarray}
\left|\int\limits_{x-\frac{x}{\sqrt{u_n}}}^x g'_x(t)I_n(x,t)dt \right|&\leq & \int\limits_{x-\frac{x}{\sqrt{u_n}}}^x |g'_x(t)||I_n(x,t)|dt\nonumber\\
&\leq & \int\limits_{x-\frac{x}{\sqrt{u_n}}}^x |g'_x(t)-g'_x(x)|dt,~~~~g'_x(x)~=0, (\text{where}~|I_n(x,t)|\leq 1)\nonumber\\
&\leq & \int\limits_{x-\frac{x}{\sqrt{u_n}}}^x V_t^xg'_x dt\nonumber\\&\leq & V_{x-\frac{x}{\sqrt{u_n}}}^xg'_x\int\limits_{x-\frac{x}{\sqrt{u_n}}}^x dt\nonumber\\
&=&\frac{x}{\sqrt{u_n}}\left(V_{x-\frac{x}{\sqrt{u_n}}}^xg'_x \right)
\end{eqnarray}

Using  Lemma \ref{lem2} for solving second term by substituting $t=x-\frac{x}{u}$,  we get

\begin{eqnarray}
\int\limits_x^{x-\frac{x}{\sqrt{u_n}}} |g'_x(t)|I_n(x,t)dt &\leq & \frac{2\zeta_n^2(x)}{u_n} \int\limits_x^{x-\frac{x}{\sqrt{u_n}}} \frac{|g'_x(t)|}{(x-t)^2} dt\nonumber\\
&\leq & \frac{2\zeta_n^2(x)}{u_n} \int\limits_x^{x-\frac{x}{\sqrt{u_n}}} V_t^xg'_x\frac{1}{(x-t)^2}dt\nonumber\\
&=& \frac{2\zeta_n^2(x)}{xu_n} \int\limits_x^{\sqrt{u_n}} V_{x-\frac{x}{u}}^sg'_x du\nonumber\\
&\leq & \frac{2\zeta_n^2(x)}{xu_n} \sum\limits_{j=0}^{[\sqrt{u_n}]}\left(V_{x-\frac{x}{j}}^tg'_x \right).
\end{eqnarray}

Hence, 

\begin{eqnarray}
|P_1|\leq \frac{2\zeta_n^2(x)}{xu_n} \sum\limits_{j=0}^{[\sqrt{u_n}]}\left(V_{x-\frac{x}{j}}^tg'_x \right)+\frac{x}{\sqrt{u_n}}\left(V_{x-\frac{x}{\sqrt{u_n}}}^xg'_x \right).
\end{eqnarray}

To solve $P_2$, we reform $P_2$ and integrating by parts, we have

\begin{eqnarray*}
|P_2|&=&\Bigg| \int\limits_x^z\left(\int\limits_x^tg'_x(u)du \right)\frac{\partial}{\partial t}(1-I_n(x,t))dt + \int\limits_z^\infty\left(\int\limits_x^tg'_x(u)du \right)\frac{\partial}{\partial t}(1-I_n(x,t))dt\Bigg| \\
&\leq & \left|\int\limits_x^z\left(\int\limits_x^tg'_x(u)du \right)\frac{\partial}{\partial t}(1-I_n(x,t))dt  \right| + \left| \int\limits_z^\infty\left(\int\limits_x^tg'_x(u)du \right)\frac{\partial}{\partial t}(1-I_n(x,t))dt     \right|\\
&=& \Bigg|  \left[\int\limits_x^tg'_x(u)du (1-I_n(x,t))\right]_x^z-\int\limits_x^z g'_x(t)(1-I_n(x,t))dt \\
&& +  \left[\int\limits_x^tg'_x(u)du (1-I_n(x,t)) \right]_z^\infty - \int\limits_z^\infty g'_x(t)(1-I_n(x,t))  dt\Bigg| \\
&=& \Bigg| \int\limits_x^z g'_x(u)du (1-I_n(x,z)) -\int\limits_x^z g'_x(t)(1-I_n(x,t))dt \\
&&  -\int\limits_x^z g'_x(u)du (1-I_n(x,z))-\int\limits_z^\infty g'_x(t)(1-I_n(x,t))  dt \Bigg|\\
&=& \Bigg| \int\limits_x^z g'_x(t)(1-I_n(x,t))dt +\int\limits_z^\infty g'_x(t)(1-I_n(x,t))  dt\Bigg|\\
&\leq &  \int\limits_x^z V_x^t (g'_x) dt+ \frac{2\zeta_n^2(x)}{u_n}\int\limits_z^\infty V_x^t(g'_x) \frac{1}{(t-x)^2}dt\\
&\leq & \frac{x}{\sqrt{u_n}} V_x^{x+\frac{x}{\sqrt{u_n}}}(g'_x)+\frac{2\zeta_n^2(x)}{u_n}\int\limits_{x+\frac{x}{\sqrt{u_n}}}^\infty V_x^t(g'_x) \frac{1}{(t-x)^2}dt.
\end{eqnarray*}

On substituting $t=x\left(1+\frac{1}{\beta} \right)$, we obtain

\begin{eqnarray*}
|P_2|&\leq & \frac{x}{\sqrt{u_n}} V_x^{x+\frac{x}{\sqrt{u_n}}}(g'_x)+\frac{2\zeta_n^2(x)}{xu_n}\int\limits_{0}^{\sqrt{u_n}} V_x^{x+\frac{x}{\beta}}(g'_x) d\beta\\
&\leq & \frac{x}{\sqrt{u_n}} V_x^{x+\frac{x}{\sqrt{u_n}}}(g'_x)+ \frac{2\zeta_n^2(x)}{xu_n} \sum\limits_{j=0}^{[\sqrt{u_n}]}\int\limits_{j}^{\sqrt{j+1}}V_x^{x+\frac{x}{j}}(g'_x) d\beta\\
&=& \frac{x}{\sqrt{u_n}} V_x^{x+\frac{x}{\sqrt{u_n}}}(g'_x)+ \frac{2\zeta_n^2(x)}{xu_n} \sum\limits_{j=0}^{[\sqrt{u_n}]} V_x^{x+\frac{x}{j}}(g'_x).
\end{eqnarray*}

Using the value of $P_1, P_2$ in (\ref{bv8}), we obtain 
\begin{eqnarray}\label{bv9}
 \nonumber \int\limits_0^\infty Y_n(x,t)\left(\int\limits_x^s(g'_x(u))~du\right)dt &=& \frac{2\zeta_n^2(x)}{xu_n} \sum\limits_{j=0}^{[\sqrt{u_n}]}\left(V_{x-\frac{x}{j}}^tg'_x \right)+\frac{x}{\sqrt{u_n}}\left(V_{x-\frac{x}{\sqrt{u_n}}}^xg'_x \right)\\
 && + \frac{x}{\sqrt{u_n}} V_x^{x+\frac{x}{\sqrt{u_n}}}(g'_x)+ \frac{2\zeta_n^2(x)}{xu_n} \sum\limits_{j=0}^{[\sqrt{u_n}]} V_x^{x+\frac{x}{j}}(g'_x)
\end{eqnarray}
Put the above value from (\ref{bv9}) in (\ref{bv7}), we obtain required result

\begin{eqnarray*}
|B_n^*(g;x)-g(x)|&\leq &  \frac{1}{2u_n}|g'(x+)+g'(x-)|+ \sqrt{\frac{1}{2u_n}}|g'(x+)-g'(x-)|\zeta_n(x)\\
&& +\frac{2\zeta_n^2(x)}{xu_n} \sum\limits_{j=0}^{[\sqrt{u_n}]}\left(V_{x-\frac{x}{j}}^tg'_x \right)+\frac{x}{\sqrt{u_n}}\left(V_{x-\frac{x}{\sqrt{u_n}}}^xg'_x \right)\\
 && + \frac{x}{\sqrt{u_n}} V_x^{x+\frac{x}{\sqrt{u_n}}}(g'_x)+ \frac{2\zeta_n^2(x)}{xu_n} \sum\limits_{j=0}^{[\sqrt{u_n}]} V_x^{x+\frac{x}{j}}(g'_x).
\end{eqnarray*}

\section{Graphical and numerical analysis of the operators}
In this section, we study the graphical representation  and numerical analysis of the operators to the function.
\begin{example}
Let the function $g:[0,2.5]\to [0,\infty)$ such that $g(x)=-x^3e^{-5x}$(blue) for all $x\in [0,2.5].$ Choosing $u_n=n=15, 35, 50$ and then corresponding operators are $S_{15}^*(g;x), S_{35}^*(g;x), S_{50}^*(g;x)$ represent  green, red and black colors respectively in the given Figure \ref{F1}. One can observe that as the value of $n$ is increased, the error of the operators to the function is going to be least. We can say that the approach of the operators to the function is good for the large value of $n$.\\

But for the same function, if we move towards the truncation type error, we can observe by Figure \ref{F2}, the approximation is not better throughout the interval $[0,2.5]$. Here we consider the $u_n=n=15, 35, 50$ and $j=15, 35, 50$, using these values, the truncation is determined. So one can observe that at a some stage, its going good but not at all.

\begin{figure}[h!]
    \centering 
    \includegraphics[width=.52\textwidth]{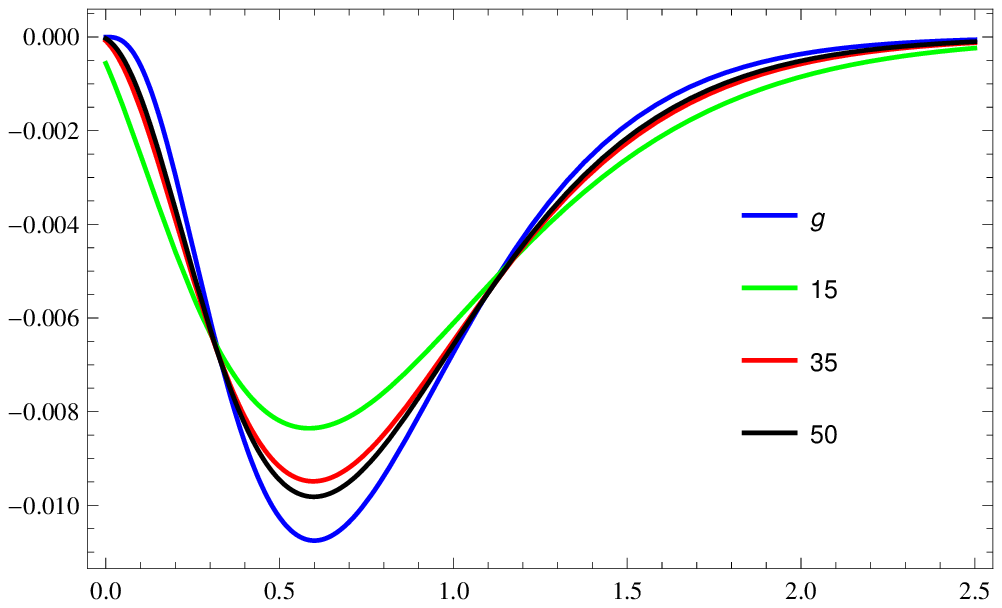}   
    \caption[Description in LOF, taken from~\cite{source}]{The convergence of the operators $S_{n}^*(g;x)$ to the function $g(x)(blue)$.}
    \label{F1}
\end{figure}
\begin{figure}[h!]
    \centering 
    \includegraphics[width=.52\textwidth]{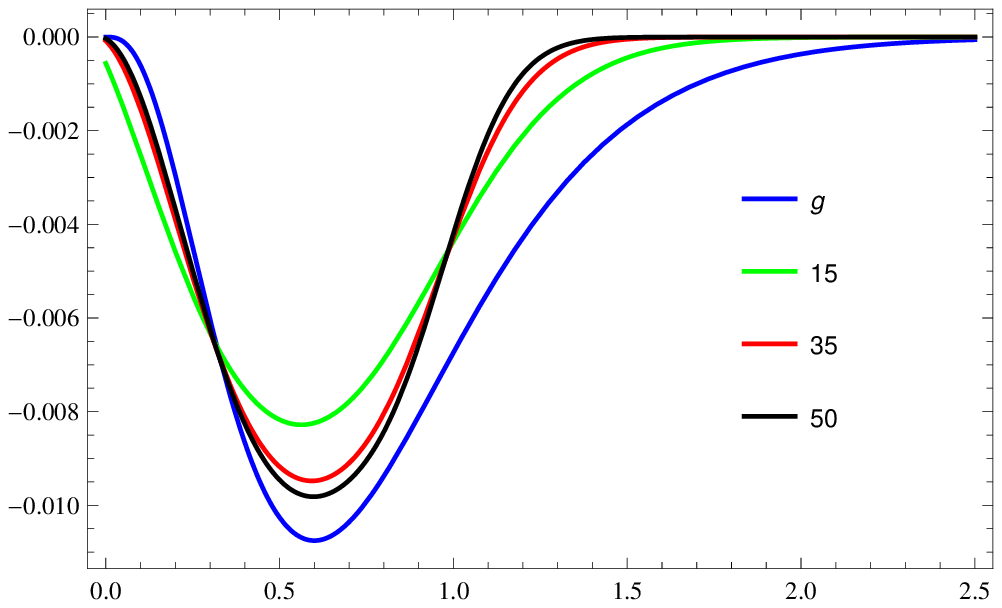}   
    \caption[Description in LOF, taken from~\cite{source}]{The convergence of the operators $S_{n}^*(g;x)$ to the function $g(x)(blue)$.}
    \label{F2}
\end{figure}
\end{example}
\pagebreak
Now, we determine the convergence of the operators to the function by considering the different sequences for the operators and then we see that the variation of the convergence to the function is changed.

\begin{example}
Let the function $g(x)=x^2e^{2x}$(black), for all $0\leq x\leq 2.5$. Here, we consider $u_n=n$ and choosing the value of $n=10, 50, 100, 200, 250, 500, 1000$, for which the operators's curve is red for the all values of $n$. Then, we can observe the error estimations by Figure \ref{F3} as well Table \ref{t1} at different points of $x$, which is going to be better as the value of $n$ is increased.

\begin{figure}[h!]
    \centering 
    \includegraphics[width=.42\textwidth]{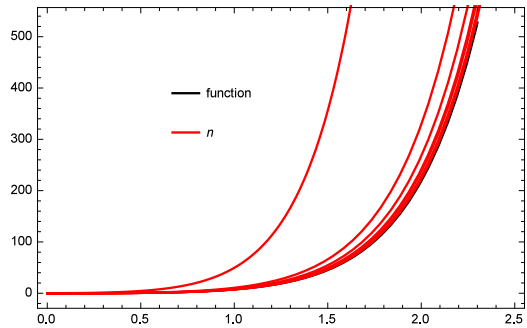}   
    \caption[Description in LOF, taken from~\cite{source}]{The convergence of the operators $S_{n}^*(g;x)$ to the function $g(x)$.}
    \label{F3}
\end{figure}
\begin{table}[ht]
\centering
\begin{tabular}{|c|c|c|c|c|c|c|c|}
\hline 
$x\downarrow$, $u_n=n\to$ & at n=10 & at n=50 & at n=100 & at n=200 & at n=250 & at n=500 & at n=1000\\
\hline 
0.1 & 0.202522 & 0.0156053 & 0.0069326 & 0.00326665 & 0.00258244 & 0.00126086 & 0.000622967 \\ 
\hline 
0.5 & 3.82396 & 0.325365 & 0.148479 & 0.0710035 & 0.0563036 & 0.0276615 & 0.0137104 \\ 
\hline 
0.9 & 27.2622 & 2.13631 & 0.969982 & 0.462837 & 0.366865 & 0.180094 & 0.0892291 \\ 
\hline 
1.0 & 42.1618 & 3.22439 & 1.46137 & 0.696735 & 0.552174 & 0.270979 & 0.134238 \\ 
\hline 
1.5 & 310.724 & 20.8491 & 9.3538 & 4.43876 & 3.51461 & 1.72172 & 0.852162 \\ 
\hline 
2.0 & 1888.96 & 110.236 & 48.9145 & 23.0939 & 18.2677 & 8.93151 & 4.4164\\ 
\hline 
2.5 & 10237.6 & 516.742 & 226.689 & 106.464 & 84.1292 & 41.0503 & 20.2783\\ 
\hline 
\end{tabular} 
\caption{Convergence estimations of the operators $B_n^*(g;x)$ to the function $g(x)$}\label{t1}
\end{table}
\end{example}
\pagebreak
\begin{example}
Let for the same function $g(x)=x^2e^{2x}$(black), for all $0\leq x\leq 2.5$. Here, we consider $u_n=n^{\frac{3}{2}}$ and choosing the value of $n=10, 50, 100, 200, 250, 500, 1000$, the curves of the operators (\ref{rb}) represent green color for all values of $n^{\frac{3}{2}}$ for the operators (\ref{rb}). Hence, we can observe the error estimations by Figure \ref{F4} as well Table \ref{t2} at the different points of $x$.
\begin{figure}[h!]
    \centering 
    \includegraphics[width=.42\textwidth]{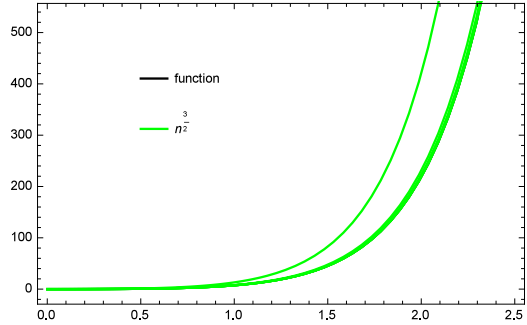}   
    \caption[Description in LOF, taken from~\cite{source}]{The convergence of the operators $S_{n}^*(g;x)$ to the function $g(x)$.}
    \label{F4}
\end{figure}
\begin{table}[ht]
\begin{tabular}{|c|c|c|c|c|c|c|c|}
\hline 
$x\downarrow$, $u_n=n^{\frac{3}{2}}\to$ & at n=10 & at n=50 & at n=100 & at n=200 & at n=250 & at n=500 & at n=1000 \\ 
\hline 
0.1 & 0.0282979 & 0.0018008 & 0.000622967 & 0.000218562 & 0.000156203 & 0.0000551185 & 0.0000194739 \\ 
\hline 
0.5 & 0.574288 & 0.0394044 & 0.0137104 & 0.0048201 & 0.00344596 & 0.0012166 & 0.000429916 \\ 
\hline 
0.9 & 3.79761 & 0.256632 & 0.0892291 & 0.0313623 & 0.0224205 & 0.00791509 & 0.00279694 \\ 
\hline 
1.0 & 5.74555 & 0.386191 & 0.134238 & 0.0471774 & 0.033726 & 0.011906 & 0.00420715 \\ 
\hline 
1.5 & 37.6466 & 2.45554 & 0.852162 & 0.299321 & 0.213959 & 0.0755213 & 0.0266852 \\ 
\hline 
2.0 & 201.92 & 12.7484 & 4.4164 & 1.55031 & 1.10808 & 0.391058 & 0.138172 \\ 
\hline 
2.5 & 960.667 & 58.6418 & 20.2783 & 7.11386 & 5.08411 & 1.79398 & 0.633827 \\ 
\hline 
\end{tabular} 
\caption{Convergence estimations of the operators $B_n^*(g;x)$ to the function $g(x).$}\label{t2}
\end{table}
\end{example}

\pagebreak
\begin{example}
Further for the function $g(x)=x^2e^{2x}$(black), for all $0\leq x\leq 2.5$, one can see the error estimations of the operators (\ref{rb}). Here, we consider $u_n=n^{2}$ and choosing the value of $n=10, 50, 100, 200, 250, 500, 1000$, the curves of the operators (\ref{rb}) represent Magenta color for all values of $n^2$ of the operators. Hence, we can observe the error estimations by Figure \ref{F5} as well Table \ref{t3} at different points of $x$. By observing, we can see, the function's curve almost overlapped by the curves of the operators.

\begin{figure}[h!]
    \centering 
    \includegraphics[width=.42\textwidth]{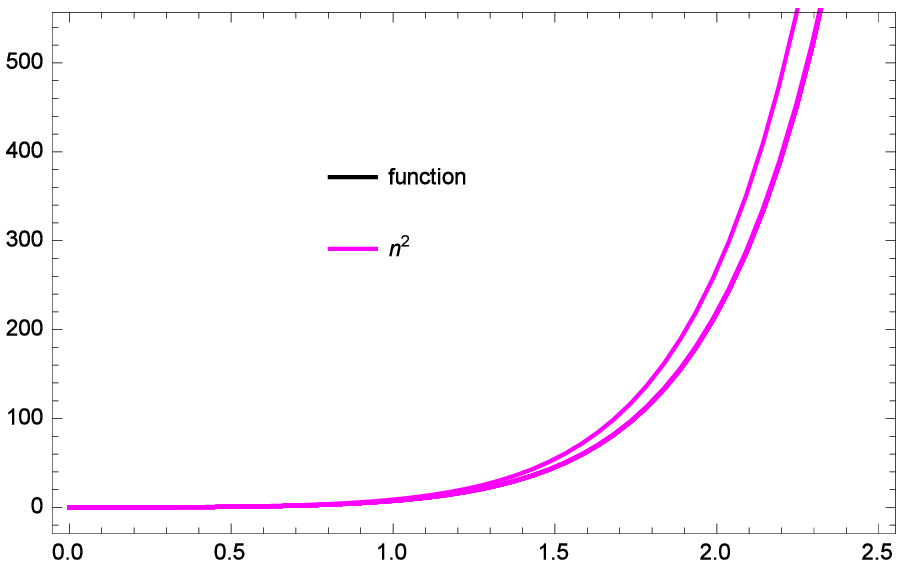}   
    \caption[Description in LOF, taken from~\cite{source}]{The convergence of the operators $S_{n}^*(g;x)$ to the function $g(x)$.}
    \label{F5}
\end{figure}
\begin{table}[ht]
\begin{tabular}{|c|c|c|c|c|c|c|c|}
\hline 
 $u_n=n^2\to$ & at n=10 & at n=50 & at n=100 & at n=200 & at n=250 & at n=500 & at n=1000 \\ 
\hline 
$x$=0.1 & 0.0069326 & 0.000247412 & 0.0000616321 & 0.0000153943 & 9.85127$\times$10$^{-6}$ & 2.462477$\times$10$^{-6}$ & 6.15594$\times$10$^{-7}$ \\ 
\hline 
$x$=0.5 & 0.148479 & 0.00545553 & 0.00136032 & 0.000339859 & 0.000217493 & 0.0000543675 & 0.0000135915 \\ 
\hline 
$x$=0.9 & 0.969982 & 0.0354973 & 0.00885019 & 0.00221104 & 0.00141495 & 0.000353699 & 0.0000884224 \\ 
\hline 
$x$=1.0 & 1.46137 & 0.053398 & 0.0133126 & 0.00332584 & 0.00212836 & 0.000532032 & 0.000133004 \\ 
\hline 
$x$=1.5 & 9.3538 & 0.338802 & 0.0844444 & 0.0210951 & 0.0134997 & 0.00337451 & 0.000843601 \\ 
\hline 
$x$=2.0 & 48.9145 & 1.75487 & 0.437267 & 0.109226 & 0.069898 & 0.0174722 & 0.0043679 \\ 
\hline 
$x$=2.5 & 226.689 & 8.0529 & 2.00598 & 0.501045 & 0.320634 & 0.080147 & 0.020036 \\ 
\hline 
\end{tabular} 
\caption{Convergence estimations of the operators $B_n^*(g;x)$ to the function $g(x)$}\label{t3}
\end{table}
\end{example}
\begin{example}
At the same time for the same function $g(x)=x^2e^{2x}$, $0\leq x\leq 2.5$, we can observe by the given Figure \ref{F6} that the accuracy of the convergence for the operators (\ref{rb}) is better when $u_n=n^2$ is taken rather than when we choose the sequences $u_n=n$ and $u_n=n^{\frac{3}{2}}$ for the same operators (\ref{rb}).
\begin{figure}[h!]
    \centering 
    \includegraphics[width=.42\textwidth]{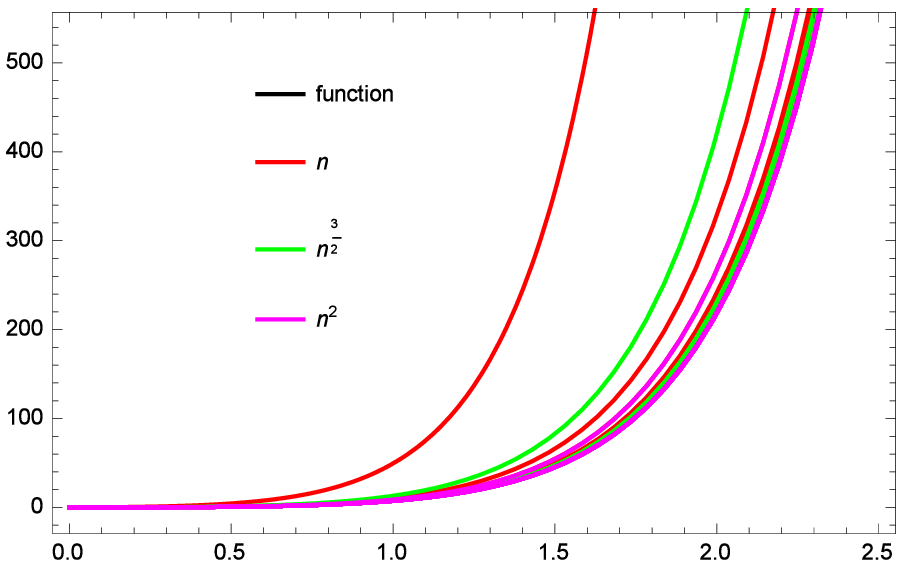}   
    \caption[Description in LOF, taken from~\cite{source}]{The convergence of the operators $S_{n}^*(g;x)$ to the function $g(x)$.}
    \label{F6}
\end{figure}
\end{example}
\pagebreak
\textbf{Concluding Remark:} After observing by all the Figures (\ref{F1},\ref{F2},\ref{F3},\ref{F4},\ref{F5},\ref{F6}) and Tables (\ref{t1}, \ref{t2}, \ref{t3}), we can conclude that the better approximation can be obtained by choosing the appropriate sequence for the operators (\ref{rb}) and in addition will get good approximation for the large value of $n$ of the  positive and real sequence for the operators.\\

\textbf{Conclusion:}
The approximation properties have been determined for the functions belonging to different spaces and moreover the rate of the convergence of the operators  has been discussed. Not in theoretical sense to support of our approximation results but also using graphical means, we presented the graphical analysis.

\end{document}